\documentclass[12pt,reqno]{amsart}
\usepackage{amssymb}

\usepackage{amscd}

\newcommand{\RNum}[1]{\uppercase\expandafter{\romannumeral #1\relax}}

\usepackage{graphicx}

\usepackage[T2A]{fontenc}
\usepackage[utf8]{inputenc}
\usepackage[russian,english]{babel}
\input{int.def}

\usepackage{mdwlist}

\usepackage{hyperref}
\usepackage{blindtext}
\usepackage{scrextend}

\usepackage{tikz}
\usetikzlibrary{matrix,arrows,calc}
\usepackage{tikz-cd}
\usepackage{graphicx}
\usepackage{caption}
\usepackage{subcaption}
\usepackage{wrapfig}

\usepackage{amsmath}

\usepackage{xcolor}
\usepackage{array}

\usepackage{tikz-cd}
\usetikzlibrary{cd}

\usepackage{comment}
\usepackage{enumitem}

\numberwithin{equation}{section}

\myitemmargin 
\def\arg{\sf{Arg}}

\def\rp{\mathbb{RP}}

\def\aut{\sf{Aut}}
\def\Gr{\sf{Gr}}
\def\rk{\sf{rk}}

\def\sia{$\sigma$-anti-invariant}
\def\ss{\mib{s}}

\def\remark{\stateskip\noindent{\bf Remark. \;}} 

\usepackage{geometry}

\newgeometry{vmargin={25mm}, hmargin={15mm,19mm}}   

\usepackage{titlesec}
\titleformat{\section}[runin]{\bfseries}{\thesection.}{3pt}{}[.]

\usepackage{mathabx,graphicx}

\begin{document}

\myitemmargin

\title[
Infinitely many non-isotopic real symplectic forms 
on $S^2 \times S^2$
]%
{Infinitely many non-isotopic real symplectic forms 
on $S^2 \times S^2$}

\author{Gleb Smirnov}


\dedicatory{}


\begin{abstract}
Let $(S^2,\omega)$ be a symplectic sphere, and let $\tau \colon S^2 \to S^2$ be an anti-symplectic involution of 
$(S^2,\omega)$. We consider the product $(S^2,\omega) \times (S^2,\omega)$ endowed with the anti-symplectic involution $\tau \times \tau$, and study the space of monotone anti-invariant symplectic forms on this four-manifold. We show that this space is disconnected. 
In addition, during the course of the proof, we produce a 
diffeomorphism of $\Gr(2,4)$ which induces the identity map on all homology and homotopy groups, but which is not homotopic to the identity.
\end{abstract}

\maketitle



\setcounter{section}{0}
\section{Introduction}Is there a closed four-manifold $X$ and a cohomology 
class $\xi \in \sfh^2(X;\rr)$ such that the space 
$\Omega_{\xi}$
of symplectic forms of class $\xi$ is connected?
This uniqueness problem up to isotopy for cohomologous symplectic forms is completely open in dimension four, though disconnected examples are known in higher dimensions, see Problem 2 and Theorem 9.7.4 in \cite{McD-Sal}.
This short note concerns the following modified version of this problem.

A real symplectic 4-manifold\footnote[1]{A somewhat unfortunate term, but let's follow 
\cite{W,Kh-Sh}.} is a triple $(X,\sigma,\omega)$, where $\omega$ is a symplectic form on $X$ and $\sigma \colon X \to X$ is an involution which 
is anti-symplectic
\[
\sigma^{*} \omega = -\omega.
\]
Pick a class $\xi \in \sfh^2(X;\rr)$ such that 
$\sigma^{*} \xi = - \xi$ and let 
$\rr\Omega_{\xi}$ denote 
the space of those symplectic forms on $X$ which are $\sigma$-anti-invariant and are in the class $\xi$. A natural question to ask is: are there any examples of 
disconnected spaces 
$\rr\Omega_{\xi}$?

Let us introduce a couple of simple examples of real symplectic manifolds.
Consider the real ruled quadric $X_1$ defined as
\begin{equation}\label{eqX1}
X_1 := \left\{
\boldsymbol{x} \in \cp^3\,|\, x_0^2 + x_1^2 = x_2^2 + x_3^2
\right\}.
\end{equation}

The set of real points of $X_1$ is a $2$-torus in $\rp^3$ 
which is doubly ruled by real 
projective lines. Being a smooth projective variety, the surface $X_1$ 
inherits a K{\"a}hler form $\omega$ from the ambient space $(\cp^3, \Omega_{st})$. 
Here $\Omega_{st}$ stands for the Fubini-Study $2$-form.
After rescaling $\omega$, we assume it is monotone meaning that
\[
[\omega] =  -K_{X_1} \in \sfh^2(X_1;\zz),
\]
where $K_{X_1}$ is the canonical class of $X_1$.
The complex conjugation $\sigma \colon \cp^3 \to \cp^3$, $\sigma(x_i) = \bar{x}_i$ is an anti-symplectic involution w.r.t. to $\Omega_{st}$. 
If a projective hypersurface is cut out by a real polynomial then it is preserved by $\sigma$ and hence itself carries an anti-symplectic involution, 
namely 
$(\sigma|_{X_1})^{*}\omega = -\omega$. In this anti-invariant setting, we answer to the question of uniqueness in the negative proving
\begin{thma}\label{thm1}
The space of monotone anti-invariant symplectic 
forms on $X_1$ has infinitely many 
connected components.
\end{thma}
\noindent
A weaker statement applies to the quadric $X_2$ defined as
\begin{equation}\label{eqX2}
X_2 := \left\{
\boldsymbol{x} \in \cp^3\,|\, x_0^2 + x_1^2 + x_2^2 + x_3^2 = 0
\right\}.
\end{equation}
In this case we have:
\begin{thma}\label{thm2}
The space of monotone anti-invariant symplectic 
forms on $X_2$ has at least two connected components.
\end{thma}
To prove both theorems, concrete representatives of different connected components 
of $\rr \Omega_{(-K_{X_1})}$ 
(for abbreviation, we let 
$\rr \Omega_{K}$ stand for 
$\rr\Omega_{(-K_{X_1})}$) will be presented; these forms will be related by a diffeomorphism and so belong to the same connected component of the moduli space of $\sigma$-anti-invariant forms. 
It is not uninteresting to compare Theorems \ref{thm1} and \ref{thm2} with the recent results of Kharlamov and Shevchishin (see \cite{Kh-Sh}), who study real symplectic 
4-manifolds up to the equivalence relation generated by deformations and diffeomorphisms.
In particular, Theorem 1.1 of \cite{Kh-Sh} states that if a 
real \textit{rational} symplectic 4-manifold $(X,\sigma,\omega)$ is 
$\sigma$-minimal, then it is a real K{\"a}hler surface. 
We do not discuss the notion of $\sigma$-minimality but note that a minimal surface (e.g. $S^2 \times S^2$) is also $\sigma$-minimal.

Although there are countably many non-isomorphic complex structures on $S^2 \times S^2$, 
we stick to the one coming from the product $\cp^1 \times \cp^1$, for it is the only complex structure which admits a K{\"a}hler form in the anti-canonical class. 
A classical result is that there are exactly four different types of anti-holomorphic 
involutions on $\cp^1 \times \cp^1$ (see Lemma 1.16 in \cite{Koll}, 
in this note we only discuss the types $Q^{2,2}$ and $Q^{4,0}$.)
Given an underlying real algebraic surface $(X,\sigma)$ and a class $\xi \in \sfh^2(X;\rr)$, the 
space of $\sigma$-anti-invariant K{\"a}hler forms in class $\xi$ is convex.
It then follows from Moser's argument that, up to isotopy, there exists exactly one 
K{\"a}hler form in a given cohomology class. The classical Moser's trick 
considers a family of symplectic forms $\omega_t$ such that 
the cohomology class of $\omega_t$ is constant and provides a family of diffeomorphisms 
$\varphi_t$ such that $\varphi_t^{*} \omega_t = \omega_0$. 
The argument is easily adapted to the case of anti-invariant forms: if the family 
$\omega_t$ consists of $\sigma$-anti-invariant forms, then the derived isotopy 
$\varphi_t$ commute with $\sigma$.

Therefore, be given a pair of rational symplectic 4-manifolds 
$(X,\sigma,\omega)$ and $(X,\sigma',\omega')$ with $X$ diffeomorphic 
to $S^2 \times S^2$ and with both $\omega$ and $\omega'$ being monotone. 
Then $(X,\sigma,\omega)$ is diffeomorphic to $(X,\sigma',\omega')$ iff 
$\sigma$ is diffeomorphic to $\sigma'$. 
In other words, the natural mapping 
\[
\diff_{K}(X,\sigma) \to \rr\Omega_{K},\quad f \to f_* \omega
\]
is surjective. Here $\diff_{K}(X,\sigma)$ stands for the subgroup 
of those diffeomorphisms which preserve $K_{X} \in \sfh^2(X;\zz)$ 
and are $\sigma$-equivariant.
Therefore, the moduli space $\rr\Omega_{K}/\diff(X,\sigma)$ 
consists of a single point, yet, according to our claim, the space 
$\rr\Omega_{K}$ itself may have infinitely many connected components.
\medskip%

Another problem worth considering in this anti-invariant setting is to construct a pair of symplectic forms which are not deformation equivalent. 
Two (\sia) symplectic forms $\omega$ and $\omega'$ are said to be deformation equivalent 
if there exists a ($\sigma$-equivariant) diffeomorphism 
$\varphi$ such that $\varphi_{*} \omega'$ and $\omega$ are connected by a path of (\sia) symplectic forms. 
In general, without taking in account any involutions, 
inequivalent symplectic forms in dimension 4 have been obtained by 
McMullen and Taubes \cite{Mc-Taub} and later by 
Smith \cite{Sm} and Vidussi \cite{Vid}. It is not immediately clear 
how to construct similar examples in the presence of an anti-holomorphic involution. 
Note however that according to \cite{Kh-Sh}, one cannot produce such examples 
in the realm of real rational symplectic 4-manifolds. 

\state Acknowledgements.  
The author thanks Sewa Shevchishin for valuable conversations about the work in this paper and 
the referee for the prompt reply and useful comments. This work is funded by 
the ETH Zurich Postdoctoral Fellowship program.

\section{Proof of Theorem \ref{thm1}}
We now describe $X_1$ in a way that visibly exhibits its complex structure. There is a projective transformation 
sending equation \eqref{eqX1} to 
\[
y_0\, y_1 = y_2\, y_3,
\]
and transforming $\sigma \colon x_i \to \bar{x}_i$ into 
$\sigma \colon (y_0,y_1,y_2,y_3) \to (\bar{y}_1, \bar{y}_0,\bar{y}_3,\bar{y}_2)$. 
The rational functions
\[
z = \frac{y_2}{y_0},\quad  w = \frac{y_3}{y_0}
\]
define a biholomorphism from $X_1$ onto $\cp^1(z) \times \cp^1(w)$. 
In the inhomogeneous coordinates $(z,w)$, the map $\sigma$ takes the form
\begin{equation}\label{sigma1}
\sigma(z,w) = \left( \bar{z}^{-1}, \bar{w}^{-1} \right).
\end{equation}
Finally, the form $\omega$ on $X_1$ splits into a product form $\omega_{\cp^1} \oplus \omega_{\cp^1}$.

We now introduce an invariant which is capable to distinguish 
between some connected components of the space 
$\rr\Omega_{K}$ of anti-invariant monotone ($\xi = -K_{X_{1}}$) symplectic forms on $X_1$. To understand this invariant, it is the easiest to start with the product K{\"a}hler form $\omega$.
We let $L$ denote the fixed point set of $\sigma$, which is the product 
of the two copies of $\rp^1$ defined, respectively, by $|z| = 1$ and $|w| = 1$.
Pick a point $p$ on $L$ and observe that there is but one smooth complex sphere passing through $p$ for each of the generators $\sfh_2(X_1;\zz) \cong \zz(A) \oplus \zz(B)$. Denote these spheres by $C_{A}$ and $C_{B}$, respectively. Notice that the curves
\[
\gamma_{A} = C_{A} \cap L,\quad \gamma_{B} = C_{B} \cap L
\]
are transversally intersecting simple closed curves in $L$, which form 
a basis for $\sfh_1(L;\zz)$. Our invariant associates to 
$\omega \in \rr\Omega_{\xi}$ the class $[\gamma_{A}] \in \sfh_1(X;\zz)$. There is no natural choice for orientaion of $\gamma_{A}$, so we orient it somehow. To see $\gamma_{A}$ is indeed an invariant for connected components of $\rr\Omega_{K}$, 
we use the following observation of Gromov:
\begin{thma}[Gromov, 2.4.A$_1$ in \cite{Gr}]\label{thm:gro}
Let $(X,\omega)$ be $S^2 \times S^2$ endowed with a product monotone form.
Then, every $\omega$-compatible almost-complex structure $J$
defines two tranversal fibrations of $X$ into $J$-holomorphic spheres and these fibratons continuously (even smoothly) depend on $J$.
\end{thma}
\noindent
Therefore, for every $\omega$-compatible almost-complex structure 
$J$ there is but one $J$-holomorphic sphere $C_{A}$ in class $A$ passing through $p$.
\smallskip%

We let $\calj_{\omega}$ denote the space of $\omega$-compatible almost-complex structures, where one can find the subspace of $\sigma$-anti-invariant ($\sigma_{*} \circ J = -J \circ \sigma_*$) structures, denoted by $\rr\calj_{\omega}$.
We have already seen that $C_{A}$ intersects $L$ by a simple closed curve in the integrable case, but we wish to see this for an arbitrary $J \in \rr\calj_{\omega}$. 
Since $\sigma$ is anti-holomorphic, it must send the curve $C_{A}$ to another $J$-curve in class $A$. 
Since both $C_{A}$ and $\sigma(C_{A})$ pass through the point $p$, they must coincide.
As such, the restriction of $\sigma \colon C_{A} \to C_{A}$ is an anti-holomorphic involution of $C_{A}$ that has a fixed point; it has, therefore, a fixed smooth circle, which exhausts the fixed points. We conclude that 
for every $J \in \rr\calj_{\omega}$ the sphere $C_{A}$ in class $A$ intersects $L$ by a closed simple curve $\gamma_{A}$.
The class $[\gamma_{A}] \in \sfh_1(L;\zz)$ does not depend 
on $J$, for the space $\rr\calj_{\omega}$ is connected 
(see Lemma \ref{lem:almost-complex} below.)
Nor it depends on the isotopy class of $\omega$, as long as
our isotopy is $\sigma$-equivariant: 
given two forms at the same connected component of $\rr\Omega_{K}$, 
we use Moser's trick to obtain a family of $\sigma$-equivariant diffeomorphisms between them, thus identifying the corresponding spaces of almost-complex structures, spaces of holomorphic spheres etc.
\begin{lem}\label{lem:almost-complex}
Let $(X,\sigma,\omega)$ be a real symplectic manifold, and let 
$\rr \calj_{\omega}$ be the space of $\omega$-compatible almost-complex structures 
which are anti-invariant under the anti-symplectic involution. 
The space $\rr \calj_{\omega}$ is non-empty and connected, and in 
fact it is contractible by Proposition 1.1 in \cite{W}. 
\end{lem}
\begin{proof}
Let $\calr_{\sigma}$ be the space of Riemannian metrics on 
$X$ which are invariant under the anti-symplectic involution. 
Clearly, the space $\calr_{\sigma}$ is convex and hence, contractible.
There is a natural embedding 
$i \colon \rr \calj_{\omega} \to \calr_{\sigma}$ defined by
\[
J \xrightarrow{i} \omega(\cdot,J \cdot)\quad \text{for $J \in \rr \calj_{\omega}$.}
\]
We will prove that $\rr \calj_{\omega}$ is a retract of $\calr_{\sigma}$. 
Since $\calr_{\sigma}$ is connected, this would imply that $\rr \calj_{\omega}$ 
is connected too.
For every $g \in \calr_{\sigma}$ there is a unique field of 
endomorphisms $A_g$ of $T_{X}$ such that 
\[
\omega(\cdot,\cdot) = g(A_g \cdot, \cdot).
\]
Since $\omega$ is $\sigma$-anti-invariant and $g$ is $\sigma$-invariant, 
it follows that $A_g$ and $\sigma$ anti-commute, i.e. 
\begin{equation}\label{eq:ag-sigma}
A_g \circ \sigma_{*} = -\sigma_{*} \circ A_g.
\end{equation}
Furthermore, as $\omega$ is skew-symmetric, so is $A_g$, i.e. 
$A_g^t = - A_g$, where $A_g^t$ is the adjoint for $A_g$ w.r.t.\,$g$. 
Set
\begin{equation}\label{eq:jg}
J_g := (-A^2_g)^{-\frac{1}{2}} A_g,
\end{equation}
where the square root in the right-hand side of 
\eqref{eq:jg} is well defined because 
$(-A_g^2)$ is $g$-self-adjoint and positive-definite. In fact, 
self-adjoint positive operators have a unique self-adjoint positive square 
root, and this is the root we pick for $(-A_g^2)$. Also, since 
$(-A_g^2)$ and $\sigma$ commute, it follows that 
\begin{equation}\label{eq:ag12-sigma}
(-A_g^2)^{\frac{1}{2}} \circ \sigma_* =  \sigma_* \circ (-A_g^2)^{\frac{1}{2}}\quad 
\text{and}\quad 
(-A_g^2)^{-\frac{1}{2}} \circ \sigma_* =  \sigma_* \circ (-A_g^2)^{-\frac{1}{2}}.
\end{equation}
Here we have used that $\sigma$ is an involution. 
Using \eqref{eq:ag-sigma} and \eqref{eq:ag12-sigma}, 
we see that $J_g$ and $\sigma$ 
anti-commute. Moreover, $J_g$ satisfies 
\begin{equation}\label{eq:jg2-id}
J_g^2 = -\id.
\end{equation}
To prove \eqref{eq:jg2-id}, it suffices to check that
\begin{equation}\label{eq:ag-sqrtag-commute}
(-A_g^2)^{-\frac{1}{2}} \circ A_g = A_g \circ (-A_g^2)^{-\frac{1}{2}}.
\end{equation}
One should pass to the complexification of 
$T_X$ to see this. Also, using \eqref{eq:ag-sqrtag-commute}, 
it is straightforward to check that $J_g$ is $\omega$-compatible. 
We conclude that $J_g \in \rr \calj_{\omega}$.
Define $u \colon \calr_{\sigma} \to \rr \calj_{\omega}$ as
\[
u(g) := J_g.
\]
It is 
easy to see that $u \circ i = \id$, so $u$ is a retraction. 
\end{proof}
\medskip%

\noindent
To prove our theorem, we shall construct a sequence of forms $\omega_{k} \in \rr\Omega_{K}$ whose invariants $[\gamma_{A}^k] \in \sfh_1(L;\zz)$ are pairwise distinct. To this end, we find a diffeomorphism $f \colon X_1 \to X_1$ such that:
\begin{enumerate}
    \item We wish $f$ to satisfy $f \circ \sigma = \sigma \circ f$,
    so that $f_{*} \omega$ would be anti-invariant. The condition also 
    implies that $f$ keeps $L$ invariant.
    
    \item The restriction of $f$ to $L$ is smoothly isotopic to 
    (a power of) the Dehn twist along the curve $\gamma_B$, so, in particular, we would have
    $f_{*}[\gamma_{A}] = [\gamma_{A}] + m\,[\gamma_{B}] \neq [\gamma_{A}]$
    in $\sfh_1(L;\zz) \cong \zz^2$ for some $m \in \zz$. 
    
    \item And, finally, we wish $f_{*} \colon \sfh_2(X_1;\zz) \to \sfh_2(X_1;\zz)$ to be the identity isomorphism. This is necessary for using the invariant $[\gamma_{A}]$, as it was defined with the class $A$. 
\end{enumerate}
Once $f$ is found, it can be used to obtain 
infinitely many non-isotopic forms $\omega_k$ by setting $\omega_{k} := f^{k}_{*} \omega$. This finishes the proof. 
\smallskip%

\noindent
To construct $f$, we first do it on the subset
\[
\cp^1(z) \times K_{\varepsilon} \subset X,\quad 
K_{\varepsilon} = \left\{ w \in \cp^1(w)\, |\, 1 - \varepsilon < |w| < 1 + \varepsilon \right\}
\]
with the formula
\[
f(z,w) := (z\, e^{m\,\arg(w)}, w)\quad \text{\normalfont{for some $m \in \zz$}}, 
\]
which can be seen as a mapping from $K_{\varepsilon}$ to $\diff_{+}(\cp^1(z))$, the group of orientation-preserving diffeomorphisms 
of $\cp^1(z)$.
We ask whether or not this mapping can be extended to a mapping defined over the whole sphere $\cp^1(w)$. 
When $m$ is even, the answer is yes since: $K_{\varepsilon}$ is homotopy equivalent to $S^1$ and $\pi_1(\diff_{+}(S^2)) = \zz_2$.
We thus constructed $f$ as an orientation-preserving 
fiberwise diffeomorphism of the fibration 
$\cp^1(z) \times \cp^1(w) \to \cp^1(w)$.

Whether $f$ is smoothly isotopic to the identity? 
If we do not impose the condition $\sigma \circ f = f \circ \sigma$ 
on the isotopy, then the answer is yes since: $\pi_2(\diff_{+}(S^2)) = 0$. Therefore, all the forms $f_{*}^{k}\omega$ are in the same connected component of $\Omega_{K}$ but not of $\rr\Omega_{K}$.

\section{Proof of Theorem \ref{thm2}}
We employ a similar model for $X_2$. 
Namely, $X_2$ is again the product $\cp^1(z) \times \cp^1(w)$, but now $\sigma$ is given by 
\begin{equation}\label{sigma2}
\sigma(z,w) = \left( -\bar{z}^{-1}, -\bar{w}^{-1} \right).
\end{equation}
Just like in \eqref{sigma1}, the action of $\sigma$ is componentwise, but 
in contrast to \eqref{sigma1}, the new involution has no fixed points.
Using Cartesian coordinates $(\boldsymbol{x},\boldsymbol{y})$ 
on $S^2 \times S^2 \subset \rr^3 \times \rr^3$, we describe $\sigma$ as
\[
\sigma(\mib{x},\mib{y}) = (-\mib{x},-\mib{y}).
\] 
Consider the diffeomorphism $f \colon X_2 \to X_2$ given by
\[
f(\mib{x},\mib{y}) := ( -\mib{x} + 
2\,\langle \mib{x},\mib{y} \rangle \,\mib{y}, \mib{y} ).
\]
Here $\langle \mib{x},\mib{y} \rangle$ stands 
for the Euclidean inner product of $\mib{x}$ and 
$\mib{y}$.
This map does nothing on the second sphere and does the reflection of the first sphere with the respect to the axis passing through the antipodal points $\mib{y}$ 
and $-\mib{y}$. 
As $f$ and $\sigma$ commute, we obtain a 
descendant self-mapping $g$ of $Z := X_2/\sigma$ and the commutative diagram
\begin{equation}\label{onlydiag}
\begin{CD}
X_2 @>{f}>>X_2 \\
@VV{p}V @VV{p}V \\
Z @>{g}>> Z,
\end{CD}
\end{equation}
where $p \colon X \to Z$ is the covering map, which identifies $(\mib{x},\mib{y})$ 
with $(-\mib{x},-\mib{y})$. 
Geometrically, $Z \cong \Gr(2,4)$, the Grassmannian of two-planes in $\rr^4$, and $X_2 \cong S^2 \times S^2$ is the corresponding Grassmannian of oriented planes.

It is interesting to look at the algebraic properties $g$. 
Taking $\pi_k$ and $\sfh_k$ of \eqref{onlydiag}, we obtain commutative diagrams of abelian groups. 
Since $f$ is homotopic to the identity (not equivariantly!), we have
\begin{equation}\label{eq:groupsf}
f_{*} = \id \colon \pi_{k}(X_2) \to \pi_{k}(X_2)\quad \text{and}\quad 
f_{*} = \id \colon \sfh_{k}(X_2) \to \sfh_{k}(X_2)\qquad
\text{for all $k$.}
\end{equation} 
Since $p \colon X_2 \to Z$ is a connected covering we also have isomorphisms
\begin{equation}\label{eq:groupsp}
p_{*} \colon \pi_{k}(X_2) \to \pi_{k}(Z)\quad 
\text{for $k \neq 1$.}
\end{equation}
From \eqref{eq:groupsf} and \eqref{eq:groupsp} we obtain
\begin{equation}\label{eq:groupsg}
g_{*} = \id \colon \pi_{k}(Z) \to \pi_{k}(Z)\quad 
\text{for all $k \neq 1$.}
\end{equation}
To prove that $g_{*} = \id$ for $k = 1$, observe that 
$\pi_1(Z) = \zz_2$ and that the only automorphism of $\zz_2$ is the identity.
\smallskip%

\noindent
Recall that 
\begin{equation}\label{eq:homZ2}
\sfh_0(Z;\zz_2) = \zz_2,\quad 
\sfh_1(Z;\zz_2) = \zz_2,\quad 
\sfh_2(Z;\zz_2) = \zz_2 \oplus \zz_2,\quad 
\sfh_3(Z;\zz_2) = \zz_2,\quad 
\sfh_4(Z;\zz_2) = \zz_2.
\end{equation}
Here the first and the last equality follow from the connectivity of $Z$. 
The group $\sfh_1$ is easy to recover since we know that $\pi_1(Z) = \zz_2$, whereas 
the group $\sfh_3$ is recovered via Poincar\'e duality. What is left to compute is $\sfh_2$. Recall that the Schubert cell decomposition of $\Gr(2,4)$ consists 
of one 0-cell, one 1-cell, two 2-cells, one 3-cell, and of a single 4-cell. 
Thus, the dimension of $\sfh_2$ is not greater than 2. 
To show that $\sfh_2$ is exactly $\zz_2 \oplus \zz_2$, we 
explicitly describe two non-homologous cycles in $Z$.

Let us introduce the diagonal sphere 
$\Delta := 
\left\{ (\mib{x},\mib{y}) \in S^2 \times S^2\,|\, \mib{x} = \mib{y} \right\}$. The sphere $\Delta$ is invariant 
w.r.t $\sigma$, as 
$\sigma(\mib{x},\mib{x}) = (-\mib{x},-\mib{x})$. 
Let $Q$ denote $p(\Delta)$, which is an embedded 
$\rp^2$ in $Z$. The map 
\[
s \colon X_2 \to \Delta,\quad s(\mib{x},\mib{y}) := (\mib{x},\mib{x})
\]
fits into the diagram 
\begin{equation*}
\begin{CD}
X_2 @>{s}>>\Delta \\
@VV{\sigma}V @VV{\sigma}V \\
X_2 @>{s}>> \Delta,
\end{CD}
\end{equation*}
and hence induces a map $Z \to Q$ which 
we denote by the same letter $s$. Note that 
$s \colon Z \to Q$ is a fiber bundle over $Q$. 
This bundle has a tautological section, 
given by $Q$ itself. Let $F$ be any fiber of $s$. 
We claim that the classes $[F], [Q] \in \sfh_2(Z;\zz_2)$ 
are non-zero and are not equal to each 
other. Indeed, since $Q$, being a section, 
intersects $F$ at exactly one point, we have
\[
[F] \cdot [Q] = 1.
\]
This implies that both $F$ and $Q$ are not homologically trivial. They are also 
not homologous to each other, as for if they were that would imply the equality 
$[F] \cdot [Q] = [F]^2$. However, the cycle $F$, being a fiber, must have self-intersection number 0.

We claim that 
$g_{*} = \id \colon \sfh_k(Z;\zz_2) \to \sfh_k(Z;\zz_2)$ 
for all $k$. This is obvious for $k \neq 2$ for dimension reasons. To prove that for $k = 2$, observe that $g$ and $s$ commute and that $g$ keeps $Q$ fixed. 
Hence, we must have $g_{*} [F] = [F]$, 
$g_{*} [Q] = [Q]$. 

A slightly more subtle computation reveals:
\begin{equation}
\sfh_0(Z;\zz) = \zz,\quad 
\sfh_1(Z;\zz) = \zz_2,\quad 
\sfh_2(Z;\zz) = \zz_2,\quad 
\sfh_3(Z;\zz) = 0,\quad 
\sfh_4(Z;\zz) = \zz,
\end{equation}
with $\sfh_2(Z;\zz)$ generated by $[F]$. 
Clearly, the action of $g$ on 
$\sfh_k(Z;\zz)$ is also trivial for all $k$.

We conclude that $g$ induces the identity morphisms on all homology and homotopy groups. Nevertheless, we claim that $g \colon Z \to Z$ is not homotopic to the identity.
To see this we need to introduce the mapping torus 
\[
\sf{T}_{\!g} := Z \times [0,1] / \left\{ (z,1)\sim(g(z),0) \right\}.
\]
We will show now that the Stiefel-Whitney class 
$w_3(\sf{T}_{\!g})$ doesn't vanish; by contrast, the mapping torus $\sf{T}_{\!\id}$ has vanishing $w_3$. Therefore, $g$ is not homotopic to the identity.
\begin{lem}
$w_3(\sf{T}_{\!g}) \neq 0.$
\end{lem}
\begin{proof}
Remember that $g|_{Q} = \id$. Therefore, there is 
a natural embedding
\[
Q \times S^1 = Q \times [0,1]/\left\{ (q,1) \sim (g(q),0) \right\} \quad
\text{into $\sf{T}_{\!g}$.}
\]
We refer to this embedded copy of $Q \times S^1$ as $\calr$.
It suffices to prove that the restriction of 
$w_3(\sf{T}_{\!g})$ on $\calr$ does not vanish.
Let us look at the restriction of the tangent bundle to 
$\sf{T}_{\!g}$ on $\calr$. It splits as 
\[
N_{\calr/\sf{T}_{\!g}} \oplus T_{\calr},
\]
where $N_{\calr/\sf{T}_{\!g}}$ is the normal bundle to $\calr$ in $\sf{T}_{\!g}$ and $T_{\calr}$ is the tangent bundle to $\calr$. 
As $\calr = Q \times S^1$, we have
\begin{equation}\label{eq:calTr}
T_{\calr} = T_{Q} \oplus \underline{\rr},
\end{equation}
where $T_{Q}$ is the pull-back of the 
tangent bundle to $Q$ under the projection $\calr \to Q$, 
and $\underline{\rr}$ is a trivial line bundle. 

Now, we will prove the formula
\begin{equation}\label{eq:NR}
N_{\calr/\sf{T}_{\!g}} = N_{Q/Z} \otimes L_{S^1}.
\end{equation}
Here $L_{S^1}$ is the non-trivial (non-orientable) 
line bundle over $S^1$, and 
$N_{Q/Z}$ is the normal bundle to $Q$ in $Z$. 
In the right-hand side of the formula both 
factors are interpreted as the pull-backs w.r.t. the 
projections $\calr \to Q$ and $\calr \to S^1$.

Observe that, although $g|_{Q} = \id$, the derivative 
$\text{d}\,g \colon N_{Q/Z} \to N_{Q/Z}$ is not the identity but
\begin{equation*}
\text{d}\,g = -\id \colon N_{Q/Z} \to N_{Q/Z}.
\end{equation*}
Now \eqref{eq:NR} follows from a general fact. 
Suppose we have a topological space $X$ with a
vector bundle $I \colon E \to X$, and 
that we form 
a twisted bundle $I^{\circlearrowleft}$ 
over $X \times S^1$ with the total space
\[
E \times [0,1]/\left\{ 
(\mib{e},x,1) \sim (-\mib{e},x,0)
\right\},
\]
where a triple $(\mib{e},x,1)$ consists of 
$x \in X \times \left\{ 1 \right\}$ and 
$\mib{e} \in I^{-1}(x)$, 
and likewise 
for $(\mib{e},x,0)$.
Then it is easy to see that 
\[
I^{\circlearrowleft} = I \otimes L_{S^1},
\]
and as $N_{\calr/\sf{T}_{\!g}} = N_{Q/Z}^{\circlearrowleft}$, 
formula \eqref{eq:NR} follows.

Now, we will establish an isomorphism
\begin{equation}\label{eq:tautau}
N_{Q/Z} \cong T_{Q}.
\end{equation}
This is again a general fact. Suppose we are given 
a manifold $Y$ and a free involution $\tau \colon Y \to Y$. 
Let us consider the diagonal 
$\Delta = \left\{ (y_1,y_2) \in Y \times Y\,|\,y_1 = y_2 \right\}$. Set 
\[
Z := Y \times Y/\sim,\quad 
Q := \Delta/\sim,\qquad\text{where $(y_1,y_2) \sim (\tau(y_1),\tau(y_2))$.}
\]
There is a canonical way 
(there are two, choose either) to identify 
$T_{\Delta}$ with the normal bundle to $\Delta$ in 
$Y \times Y$. This identification is 
equivariant w.r.t. the involution 
$\tau \times \tau \colon Y \times Y \to Y \times Y$, and hence 
gives rise to an identification between 
$T_{Q}$ and $N_{Q/Z}$.

Combining \eqref{eq:calTr} with 
\eqref{eq:NR} and \eqref{eq:tautau} yields
\begin{equation}\label{eq:fatbundle}
N_{\calr/\sf{T}_{\!g}} \oplus T_{\calr} \cong 
T_{Q} \oplus \underline{\rr} \oplus (T_{Q} \otimes L_{S^1}).
\end{equation}
We wish to compute the class $w_3$ of this bundle. 
To this end, we apply the Whitney sum formula to obtain 
\begin{equation}\label{eq:whitney_sum}
w_3( T_{Q} \oplus (T_{Q} \otimes L_{S^1}) ) = 
w_1(Q) w_2(T_{Q} \otimes L_{S^1}) + 
w_2(Q) w_1(T_{Q} \otimes L_{S^1}),
\end{equation}
where $w_i(Q) = w_i(T_{Q})$ are the Stiefel-Whitney 
classes $Q$ (more precisely, their pull-backs to 
$\calr$ under the projection $\calr \to Q$), and 
$w_i(T_{Q} \otimes L_{S^1})$ are the classes of the bundle 
$T_{Q}$ twisted by $L_{S^1}$. In computing $w_3$ we have ignored the $\underline{\rr}$ component of \eqref{eq:fatbundle} because characteristic classes are invariant under taking the sum with a trivial bundle.

In general, if $L$ is a line bundle and 
$E$ is an arbitrary vector bundle, then there are closed 
expressions for both the first and the top classes of 
$E \otimes I$ in terms of $w_i(E)$ and $w_i(L)$. They are as follows:
\begin{equation}\label{eq:w_relations}
w_n(E \otimes L) = w_n(E) + 
w_{n-1}(E) w_1(L) + w_{n-2}(E) w_1(L)^2 + \ldots,\quad w_1(E \otimes L) = w_1(E) + n\,w_1(L),
\end{equation}
where $n = \rk\,E$. 
The proof of both formulas is a 
simple application of the splitting 
principle and is omitted. Since $\rk\,T_{Q} = 2$, we get
\begin{equation}\label{eq:w_QL}
w_2(T_{Q} \otimes L) = w_2(Q) + w_1(Q) w_1(L_{S^1}),\quad 
w_1(T_{Q} \otimes L) = w_1(Q).
\end{equation}
Since $\dim\,Q = 2$, we have $w_1(Q)w_2(Q) = 0$. 
Substituting \eqref{eq:w_QL} into 
\eqref{eq:whitney_sum} we obtain $w_1(Q)^2 w_1(L_{S^1})$. 
Recall that $Q \cong \rp^2$ and $\calr \cong Q \times S^1$. 
The class $w_1(\rp^2)^2$ generates 
$\sfh^2(\rp^2;\zz_2)$, while the class $w_1(L_{S^1})$ generates 
$\sfh^1(S^1;\zz_2)$. Clearly, their product does not vanish.
\end{proof}
\begin{lem}
$w_3(\sf{T}_{\!\id}) = 0.$
\end{lem}
\begin{proof}
Since $\mathsf{T}_{\!\id} \cong Z \times S^1$, the Whitney formula yields
\[
w_3( \sf{T}_{\!\id} ) = w_3(Z).
\]
One of the definitions of $w_3$ says
\[
w_3(Z) \cdot [Z] = \chi(Z)\,\text{$\sf{mod}$}\,2.
\]
Since there is a double covering $S^2 \times S^2 \to Z$, it follows that
\[
\chi(Z) = \dfrac{\chi(S^2 \times S^2)}{2} = 2.
\]
\end{proof}
\remark
Suppose that $X$ is a simply-connected $4$-manifold.
Then a classical result of Quinn \cite{Quinn} states that a homeomorphsim $f \colon X \to X$ is homotopic to the identity iff it induces the identity morphism on homology. The theorem's generalization to the non-simply-connected case is unknown, not even for the simplest case of the group $\zz_2$.
Suppose that $Z$ is a closed $4$-manifold with $\pi_1(Z) \cong \zz_2$.
Of course, for a homeomorphism $g \colon Z \to Z$ to be homotopic to the identity, it is first necessary to have $g$ acting identically on all $\pi_{k}$ and $H_k$. However, this is not sufficient.
For example, a simple obstruction is that any self-homeomorphism of $Z$ that is homotopic to the identity must have a lift (to the universal cover) that is also homotopic to the identity, 
but even that in general is not sufficient, as the example above shows.
\smallskip%

\noindent
Thus, although $f$ is isotopic to the identity as a diffeomorphism 
of $X_2$, it is not isotopic to the identity within 
$\diff(X_2,\sigma)$.
\smallskip%

\begin{lem}
Suppose that 
$h \colon (X_2, \sigma,\omega) \to (X_2, \sigma,\omega)$ is a 
symplectomorphism such that 
$h \circ \sigma = \sigma \circ h$ and such that 
$h_{*} = \id \colon \sfh_2(X_2;\zz) \to \sfh_2(X_2;\zz)$. Then $h$ is isotopic to the identity 
within $\diff(X_2,\sigma)$.
\end{lem}
\begin{proof}
Let $J_0$ be the underlying complex structure of the 
K{\"a}hler surface $X_2$. 
Since $h$ is a symplectomorphism 
and $J_0 \in \rr\calj_{\omega}$, it follows that $h_{*} J_0 \in \rr\calj_{\omega}$. 
Since $\rr\calj_{\omega}$ is connected, there is a path 
$J(t) \in \rr\calj_{\omega}$, $t \in [0,1]$ such that 
$J(0) = J_0$ and $J(1) = h_{*} J_0$.
Recall that every structure $J(t) \in \rr\calj_{\omega}$ gives rise to a pair 
$(\calf_{A}^{\,t},\calf_{B}^{\,t})$ of transversal fibrations 
of $X_2$ into $J(t)$-holomorphic spheres. Since $\sigma$ is anti-holomorphic for 
each $J(t)$, these fibrations are invariant 
w.r.t. $\sigma$ in the following sense: if $C$ is a fiber of $\calf^{\,t}_{A}$, 
then so is $\sigma(C)$, and likewise for $\calf^{\,t}_{B}$.
Following \cite[2.4.A$_1$]{Gr}, one constructs a path 
of diffeomorphisms $\alpha(t) \in \diff(X_2,\sigma)$ with $\alpha(0) = \id$ such 
that $\alpha(t)$ sends $(\calf_{A}^{\,0},\calf_{B}^{\,0})$ to 
$(\calf_{A}^{\,t},\calf_{B}^{\,t})$. Composing $h$ with $\alpha(t)$, we can assume 
henceforth that $h$ preserves $(\calf_{A}^{\,0},\calf_{B}^{\,0})$. 
This is a very restrictive condition implying 
that $h$ is componentwise in the sense that
\[
h = h_1 \times h_2 \colon S^2 \times S^2 \to S^2 \times S^2,
\]
where both diffeomorphisms 
$h_i \colon S^2 \to S^2$ must be orientation-preserving 
and must also commute with the antipodal map 
$\tau \colon S^2 \to S^2,\ \tau(\mib{x}) = -\mib{x}$. Here the condition $h_* = \id$ has been used twice, once 
to ensure $h$ does not interchange the fibrations $\calf_{A}^{\,0}$ and $\calf_{B}^{\,0}$, and 
once to ensure that $h$ preserves the natural complex orientation of 
their fibers. What is left is to show that the group 
\begin{equation}\label{product}
\diff_{+}(S^2,\tau) \times \diff_{+}(S^2,\tau),
\end{equation}
is connected. 
Here $\diff_{+}(S^2,\tau)$ stands for the subgroup 
of those diffeomorphisms of $S^2$ which are orientation-preserving and 
$\tau$-equivariant.

A more general statement is that $\diff_{+}(S^2,\tau)$ is homotopy 
equivalent to $\mib{SO}(3)$. 
Note that every element of $\diff(S^2,\tau)$ 
induces a self-diffeomorphism 
of 
$\rp^2 \cong S^2/\tau$, and we have a surjective 
homomorphism
\begin{equation}\label{eq:rp2}
\diff(S^2,\tau) \to \diff(\rp^2),
\end{equation}
which is not injective, since $\id \colon \rp^2 \to \rp^2$ has exactly 
two lifts, the trivial lift $\id$ and the non-trivial lift $\tau$. 
However, if we replace $\diff(S^2,\tau)$ by 
$\diff_{+}(S^2,\tau)$, then \eqref{eq:rp2} becomes an isomorphism.
One can prove that $\diff(\rp^2)$ is homotopy equivalent to $\mib{SO}(3)$ 
by following the steps of the proof of Smale's theorem (see \cite{Smale}), as 
in \cite[Appendix A]{H-P-W}. A shorter proof is as follows. 
It is a classical result that the $2$-sphere $S^2$ admits a unique complex structure. That is, the group of diffeomorphisms of $S^2$ acts 
transitively on the space of complex structures on $S^2$.
From here, it is easy to show that $\diff_{+}(S^2,\tau)$ acts transitively 
on the collection of $\tau$-anti-invariant complex structures on $S^2$, and 
that we have a fibration 
\[
\aut(\cp^1,\tau) \to \diff_{+}(S^2,\tau) \to \calj_{\tau},
\]
where $\aut(\cp^1,\tau)$ is the subgroup of those biholomorphisms 
of $\cp^2$ which are $\tau$-equivariant, and $\calj_{\tau}$ is \textit{a 
connected component} (there are two of them!) of the space of $\tau$-anti-invariant 
complex structures on $S^2$.
The space $\calj_{\tau}$ is contractible, 
while $\aut(\cp^1,\tau)$, being a Lie group, contracts on its maximal compact 
subgroup, which is $\mib{SO}(3)$. Hence, the group 
$\diff_{+}(S^2,\tau)$ itself contracts onto $\mib{SO}(3)$. 
\end{proof}
\smallskip%

We are now in a position to show $\omega$ and $f_{*} \omega$ are at different components of $\rr\Omega_{K}$. Suppose they are not. 
Then we could join them with a path of monotone 
anti-invariant forms. By Moser's trick, that would imply that 
$f$ is isotopic to a symplectomorphism within $\diff(X_2,\sigma)$, and that, 
by the previous lemma, would imply that $f$ is isotopic to the identity within
$\diff(X_2,\sigma)$. This is a contradiction.

\bibliographystyle{plain}
\bibliography{ref}

\end{document}